\documentclass[12pt]{amsart}
\numberwithin{equation}{section}
\usepackage[dvips]{graphicx}
\usepackage{amssymb}
\usepackage{amsmath}
\usepackage{latexsym}

\def\RR{\mathbb R}
\def\CC{\mathbb C}
\def\NN{\mathbb N}
\def\ZZ{\mathbb Z}

\def\dim{\operatorname{dim}}
\def\dens{\operatorname{dens}}
\def\diam{\operatorname{diam}}
\def\area{\operatorname{area}}

\def\dist{\operatorname{dist}}

\def\length{\operatorname{length}}
\def\esssup{\operatorname*{ess\,sup}}
\newtheorem{lemma}{Lemma}[section]
\newtheorem{theorem}{Theorem}
\newtheorem*{theorema}{Theorem A}
\newtheorem*{theoremb}{Theorem B}
\newtheorem*{theoremc}{Theorem C}
\newtheorem{proposition}{Proposition}[section]
\theoremstyle{definition}

\theoremstyle{remark}
\newtheorem*{rem}{Remark}
\newtheorem{remark}{Remark}

\title[Karpi\'nska's paradox]{Karpi\'nska's paradox in dimension three}
\author{Walter Bergweiler}
\thanks{Supported by the G.I.F.,
the German--Israeli Foundation for Scientific Research and
Development, Grant G-809-234.6/2003,
the EU Research Training Network CODY,
and the ESF Research Networking Programme HCAA}
\address{Mathematisches Seminar,
Christian--Albrechts--Universit\"at zu Kiel,
Lude\-wig--Meyn--Str.~4,
D--24098 Kiel,
Germany}
\email{bergweiler@math.uni-kiel.de}
\begin{document}
\begin{abstract}
For $0<\lambda<1/e$ the Julia set of $\lambda e^z$ is an uncountable
union of pairwise disjoint simple curves tending to infinity
[Devaney and Krych 1984], the Hausdorff dimension of this set is two
[McMullen 1987], but the set of curves without endpoints has Hausdorff
dimension one [Karpi\'nska 1999].
We show that these results have three-dimensional analogues
when the exponential function is replaced by a
quasiregular self-map of $\RR^3$ introduced by Zorich.
\end{abstract}
\maketitle
\section{Introduction and main result}\label{intro}

Zorich~\cite{Zorich67} has given an example of a quasiregular map
$F:\RR^3\to\RR^3\backslash \{0\}$ which in many ways can be
considered as a three-dimensional analogue of the exponential map.
In fact, the construction is quite flexible and gives a whole
class of such maps. It is the purpose of this paper to show that
certain results on the dynamics of entire functions of the form
$E_\lambda(z)=\lambda e^z$
have counterparts in the context of Zorich maps.

We first describe the results on the dynamics of the functions
$E_\lambda$ that we are concerned with. Then we briefly introduce
quasiregular maps and continue with the definition of Zorich
maps, before we finally state our results on the dynamics of such
maps.

The \emph{Julia set} $J(f)$ of an entire function~$f$
is the set of all points in~$\CC$
where the iterates $f^k$ 
of~$f$ 
do not form a normal family.
For an attracting fixed point~$\xi$ of~$f$ we call
$ A(\xi):=\left\{z:\lim_{k\to \infty}f^k(z)=\xi \right\} $
the \emph{attracting basin} of~$\xi$. It is a standard
result of complex dynamics that $\partial A(\xi)=J(f)$; 
see, e.g.,~\cite[Corollary 4.12]{Mil06}.
For an introduction to complex dynamics we refer to~\cite{Bea91,Mil06,Ste93}
for rational and to~\cite{Bergweiler93,Morosawa00} for entire functions.

For $0<\lambda<1/e$ the function $E_\lambda$ has an
attracting fixed point $\xi\in \RR$. Devaney and
Krych~\cite[p.~50]{Devaney84} proved
that $J(E_\lambda)=\CC\backslash A(\xi)$ and that $J(E_\lambda)$
is a ``Cantor set of curves'' for such~$\lambda$. To put this in a
precise form we say that a subset $H$ of $\CC$ (or $\RR^n$) is a
\emph{(Devaney) hair} if there exists a homeomorphism $\gamma:[0,\infty)\to
H$ such that $\gamma(t)\to\infty$ as $t\to\infty$. We call
$\gamma(0)$ the \emph{endpoint} of the hair~$H$. With this
terminology we obtain the following result from the work of
Devaney and Krych.

\begin{theorema} If $0<\lambda<1/e$, then $J(E_\lambda)$ is an
uncountable union of pairwise disjoint  
Devaney
hairs.
\end{theorema}

Actually the results of Devaney and Krych are much more precise by
giving additional information e.g.\ on the location of the hairs
and on the dynamics of $E_\lambda$ on them, but for simplicity we
have restricted ourselves to the above statement.

We denote by $\dim S$ the Hausdorff dimension of a subset $S$ of
$\CC$ (or $\RR^n$). The following result is due to 
McMullen~\cite[Theorem~1.2]{McMullen87}.

\begin{theoremb} 
If $\lambda \in \CC\backslash \{0\}$, 
then $\dim J(E_\lambda)=2$. 
\end{theoremb}
In the situation of Theorem~A the union of the 
hairs thus has Hausdorff dimension~$2$.
Karpi\'nska~\cite[Theorem~1.1]{Karpinska99} proved the surprising
and seemingly paradoxical result that
this changes if one removes the endpoints of the hairs.

\begin{theoremc} Let $0<\lambda<1/e$ and let $C_\lambda$ be the
set of endpoints of the Devaney hairs that form $J(E_\lambda)$. 
Then $\dim (J(E_\lambda)\backslash C_\lambda)=1$. 
\end{theoremc}
Of course, it follows from Theorem~B and C that
$\dim C_\lambda=2$. This had been proved before also by 
Karpi\'nska~\cite[Theorem~1]{Karpinska99a}.

The \emph{escaping set} 
$I(f):=\left\{z:\lim_{k\to \infty}f^k(z)=\infty\right\}$
plays an important role in complex dynamics. It was introduced by
Eremenko~\cite{Eremenko89} who showed that $J(f)=\partial
I(f)$ for every entire function~$f$.
Eremenko and Lyubich~\cite[Theorem~1]{Eremenko92} proved that 
$I(f)\subset J(f) = \overline{I(f)}$  for a large class of functions~$f$, which in
particular contains the functions~$E_\lambda$.
McMullen actually proved that~$\dim I(E_\lambda)=2$.
We mention that the results of Devaney and Krych yield that
$J(E_\lambda)\setminus C_\lambda\subset I(E_\lambda)$
while $C_\lambda$ contains points of both
$I(E_\lambda)$ and $\CC\setminus I(E_\lambda)$.

There are a large number of papers on dynamics of exponential
functions. We refer to a detailed (130 pages) survey 
by
Devaney~\cite{Devaney}, but note that the papers by Rempe~\cite{Rempe06}
and Schleicher~\cite{Schleicher03a} also contain
substantial sections devoted to a survey of the area.

We now turn to quasiregular maps. We treat this topic rather
briefly and refer to Rickman's monograph~\cite{Rickman93} for a
detailed discussion. Let $n\geq 2$ and let $G\subset \RR^n$ be a
domain. We denote the (euclidean) norm of $x\in \RR^n$ by~$|x|$. A
continuous map $f:G\to \RR^n$ is called \emph{quasiregular} if it
belongs to the Sobolev space $W^1_{n,loc}(G)$ and if there exists
a constant $K_O\geq 1$ such that
\begin{equation}\label{1a} 
|Df(x)|^n\leq K_O J_f(x)
\quad \mbox{ a.e.},
\end{equation}
where $Df(x)$ denotes the derivative,
\[
|Df(x)|:=\sup_{|h|=1} |Df(x)(h)|
\]
its norm, and $J_f(x)$ the Jacobian determinant.
With
\[
\ell(Df(x)):=\inf_{|h|=1}|Df(x)(h)|
\]
the condition that~\eqref{1a} holds 
for some $K_O\geq 1$ is equivalent to the condition that
\begin{equation}\label{1b}
J_f(x)\leq K_I\ell(Df(x))
\quad \mbox{ a.e.},
\end{equation}
for some $K_I\geq 1$. The smallest constants
$K_O$ and $K_I$ for which~\eqref{1a} and~\eqref{1b} hold are
called the \emph{outer and inner dilatation} of $f$ and
$K:=\max\{K_I,K_O\}$ is called the (maximal) \emph{dilatation} of
$f$; see~\cite{Rickman93} for more details.

An important 
example of a quasiregular map
$F:\RR^3\to \RR^3$ was given by Zorich~\cite[p.~400]{Zorich67}; see 
also~\cite[\S 6.5.4]{Iwaniec01} and~\cite[\S I.3.3]{Rickman93}. 
This map 
can be considered as a
three-dimensional analogue of the exponential function. To
describe it, we follow~\cite{Iwaniec01} and consider the
square
\[
Q:=\left\{(x_1,x_2)\in \RR^3:\; |x_1|\leq 1, |x_2|\leq 1\right\}
\]
and the upper hemisphere
\[
U:=\left\{(x_1,x_2,x_3)\in \RR^3:\; x_1^2+x_2^2+x_3^2=1,\;x_3\geq0\right\}.
\]
Let $h:Q\to U$ be a bilipschitz map and define 
\[
F:Q\times \RR \to \RR^3,\
F(x_1,x_2,x_3)=e^{x_3}h(x_1,x_2).
\]
Then $F$ maps the ``infinite square beam'' $Q\times \RR$ 
to the upper half-space. Repeated reflection along 
sides of square beams and the $(x_1,x_2)$-plane 
yields a map
$F:\RR^3\to \RR^3$. It turns out that this map $F$ is
quasiregular. Its dilatation is bounded in terms of the
bilipschitz constant of~$h$. We call a function $F$ defined this
way a \emph{Zorich map}.

We note that 
$ F(x_1+4,x_2,x_3)=F(x_1, x_2+4,x_3)=F(x_1,x_2,x_3) $
for all $x\in \RR^3$ so that~$F$ is ``doubly periodic''.

If $DF(x_1,x_2,0)$ exists, then
\begin{equation}\label{1c}
DF(x_1,x_2,x_3)=e^{x_3}DF(x_1,x_2,0),
\end{equation}
and this implies that there exist $\alpha,m,M\in \RR$ with
$0<\alpha<1$ and $m<M$ such that
\begin{equation}\label{1d}|DF(x_1,x_2,x_3)|\leq\alpha \quad 
\mbox{ a.e. for }x_3\leq m
\end{equation}
while
\begin{equation}\label{1e}\ell(DF(x_1,x_2,x_3))\geq 
\frac{1}{\alpha}\quad \mbox{ a.e. for }x_3\geq M.
\end{equation}
We now choose
\begin{equation}\label{1f}a\geq e^M-m
\end{equation}
and consider the map 
\[
f_a:\RR^3\to\RR^3,\ 
f_a(x)=F(x)-(0,0,a).
\]
We will use the term Zorich map also for the functions~$f_a$.

We now state the result which can be seen as a three-dimensional
analogue of Theorems~A, B and C.

\begin{theorem}\label{thm1}
Let $f_a$ be a Zorich map with parameter $a$
satisfying~\eqref{1f}. Then there exists a unique fixed point
$\xi=(\xi_1,\xi_2,\xi_3)$ satisfying $\xi_3\leq m$, the
set
\[
J:=\left\{x\in \RR^3:\; f_a^k(x)\not\rightarrow \xi\right\}
\]
consists of uncountably many pairwise disjoint hairs, and the set $C$
of endpoints of these hairs has Hausdorff dimension $3$ while
$J\backslash C$ has Hausdorff dimension~$1$.
\end{theorem}
For $0<\lambda<1/e$ 
the set $C_\lambda$ of endpoints of the hairs in
$J(E_\lambda)$
can also be characterized as
the set of points which are accessible from 
the attracting basin $A(\xi)$; 
see~\cite[Corollary~4.7]{Devaney87} and also~\cite{Karpinska99}. 
It turns out
that the situation is different for Zorich maps.
\begin{theorem}\label{thm2}
Let $J$ be as in Theorem~$\ref{thm1}$.
Then all points of $J$ are accessible from~$\RR^3\backslash J$.
\end{theorem}
After some preliminaries in section~\ref{prelims}, the parts of
Theorem~\ref{thm1} that correspond to Theorems~A, B and C
will be proved in sections~\ref{constructionhairs},
\ref{hausdorffhairs} and \ref{hausdorffwithout}, respectively, while
Theorem~\ref{thm2} will be proved in section~\ref{proofthm2}.
Some examples of Zorich maps will be discussed in section~\ref{examples}.
Besides the techniques introduced by Devaney and Krych~\cite{Devaney84},
Karpi\'nska~\cite{Karpinska99} and 
McMullen~\cite{McMullen87} we will also use 
(in particular in section~\ref{constructionhairs})
some methods of 
Schleicher and Zimmer~\cite{Schleicher03} who obtained analogues of Theorems~A 
and~C for general parameter 
values~$\lambda$. 

We conclude this introduction with a number of remarks.

\begin{remark}
Instead of the square
$Q=\left\{(x_1,x_2)\in \RR^3:|x_1|\leq 1, |x_2|\leq 1\right\}$
we could have taken any rectangle
$Q=\left\{(x_1,x_2)\in \RR^3:|x_1|\leq  c_1, |x_2|\leq c_2\right\}$
in the construction of $F$ and~$f_a$.
In particular, we may take $c_1=\frac12 \pi$ and 
choose the  function $h:Q\to U$ such that
$h(x_1,0)=(\sin x_1, 0,\cos x_1)$.
Then $F(x_1,0,x_3)=\left(e^{x_3} \sin x_1, 0,e^{x_3}\cos x_1\right)$.
The function~$F$ thus leaves the $(x_1,x_3)$-plane
invariant and its restriction to this plane is conjugate
to the exponential function in the plane.
In fact, with
$\phi:\left\{(x_1,0,x_3):x_1,x_3\in \RR\right\}\to\CC$, $\phi(x_1,0,x_3)=x_3+ix_1$,
we have
$\phi\circ F\circ\phi^{-1} =\exp$.
This underlines that Zorich maps can be seen as three-dimensional
analogues of the exponential function.
\end{remark}

\begin{remark}\label{rembranch}
Schleicher and Zimmer~\cite{Schleicher03} have shown that 
$I(E_\lambda)$ consists of pairwise disjoint, injective, unbounded
curves for all $\lambda\in
\CC\backslash \{0\}$. 
For general Zorich maps $f_a(x)=F(x)-a$
with $a\in \RR^3$ 
the situation is different,
since -- in contrast
to the exponential function -- Zorich maps have branch points.
In fact, the branch points
of $F$ are the edges of the square beam $Q\times \RR$ and the
lines obtained from these edges by reflection. While it is
possible to construct ``tails'' of hairs by the methods employed
in this paper, complications arise when a branch point is mapped
onto such a tail by an iterate of~$F$.

We mention that Zorich~(\cite{Zorich67}, 
see also~\cite[Corollary III.3.8]{Rickman93}) proved a conjecture
of Lavrent'ev saying that if $n\geq3$ and
$f:\RR^n\to \RR^n$ is quasiregular and not bijective, then $f$ has
branch points.  In fact, the branch set has Hausdorff
dimension at least $n-2$;
see~\cite[p.~12]{Rickman93} and~\cite{Heinonen} for further discussion.
\end{remark}

\begin{remark}
While~\cite{Schleicher03} extends many results from
the special case $0<\lambda<1/e$ to 
the general case, 
some striking differences 
between these cases
were found by
Bara\'nski, Karpi\'nska and Zdunik~\cite{BKZ}:
if $E_\lambda$ has an attracting periodic 
point of period greater than~$1$, then the 
Hausdorff dimension of the boundary of 
each Fatou component is strictly between~$1$ and~$2$.
\end{remark}

\begin{remark}
The exponential maps $E_\lambda$ have at most one attracting
periodic cycle. 
Zorich maps may have infinitely many attracting fixed points. There may also 
be saddle points and
one-dimensional attractors; see section~\ref{examples} for more details.
\end{remark}

\begin{remark}
The existence of Devaney hairs is not a special feature
of the exponential function. That such hairs exist in large
classes of entire functions was already shown by
Devaney and Tangerman~\cite{Devaney86} in 1986,
and more recently for considerably wider
classes by Bara\'nski~\cite[Theorem~C]{Bar07}
and by Rottenfu{\ss}er, R\"uckert, Rempe and 
Schleicher~\cite[Theorem~1.2]{RRRS}.
We note that Bara\'nski~\cite{Bar08} has in fact shown
that analogues to Theorems~A, B and C hold if the exponential function is
replaced by an entire functions of finite order
for which the set of singularities of the inverse is bounded.
\end{remark}

\begin{remark}
Karpi\'nska's paradox
becomes even more striking 
in the sine family; that is, for maps $f$
of the form
$f(z)=\sin(\alpha z +\beta)$ where $\alpha,\beta\in\CC$, $\alpha\neq 0$.
Devaney and Tangerman~\cite[Theorem~4.1]{Devaney86} showed that 
for suitable parameters
the  Julia set consists  of
hairs and
McMullen~\cite[Theorem~1.1]{McMullen87} proved that the Julia set
and the escaping set have positive measure for maps in the sine
family.
Karpi\'nska~\cite[Theorem~3]{Karpinska99a} then proved that 
already the set of endpoints of the hairs has positive measure.
A particularly strong form of Karpi\'nska's paradox
was obtained by Schleicher~\cite{Schleicher07}
for postcritically finite maps in the sine family:
the union of
the hairs without endpoints has Hausdorff dimension~$1$, but
every point in the plane which is not in this union is the
endpoint of at least one hair.

In dimension greater than~$2$
it will be more difficult to construct postcritically finite
quasiregular maps
since here, as already
noted in Remark~\ref{rembranch}, the branch set is much larger.
\end{remark}

\begin{remark}
Mayer~\cite{Mayer90} has shown that if $0<\lambda<1/e$, then the set 
$C_\lambda$ of endpoints of the Devaney hairs is totally disconnected 
while $C_\lambda\cup\{\infty\}$ is connected. 
There are a number of other striking topological phenomena connected to the 
dynamics of exponential functions, see~\cite{Devaney} for a survey.
It is to be expected that the dynamics of Zorich maps 
are also interesting from the topological point of view.
\end{remark}

\begin{remark}
A quasiregular map $f$ is called \emph{uniformly quasiregular}
if the dilatation of the iterates $f^k$ has an upper bound which
does not depend on~$k$. For uniformly quasiregular maps
$f:\overline{\RR^n}\to\overline{\RR^n}$, where
$\overline{\RR^n}=\RR^n\cup \{\infty\}$, an iteration
theory in the spirit of Fatou and Julia has been developed
by Hinkkanen, Martin, Mayer and others~\cite{Hinkkanen04,Martin97,Mayer97};
see~\cite[Chapter~21]{Iwaniec01} for an introduction. In
principle it would also be possible to develop such a theory for
uniformly quasiregular maps $f:\RR^n\to \RR^n$. However,
for $n\geq 3$ no examples of such maps which
do not extend to quasiregular self-maps of $\overline{\RR^n}$
are known. Uniformly
quasiregular self-maps of $\RR^2$ (or $\overline{\RR^2}$) are
quasiconformally conjugate to entire (or rational) maps~\cite{Geyer94,
Hinkkanen96,Kisaka08}. There are only few  results on the dynamics of
quasiregular self-maps of~$\RR^n$ with $n\geq 3$:
in~\cite{Bergweiler08} it is proved that
if $f$ has an essential singularity at~$\infty$,
then $I(f)\neq\emptyset$. In fact, $I(f)$ has an unbounded
component for such maps~$f$.
(For entire~$f$ this was proved in~\cite{RS}.)
We also mention~\cite{Siebert06} where it is shown that
quasiregular self-maps of~$\RR^n$ with an essential singularity
at $\infty$ have periodic points of all periods greater
than~$1$.
\end{remark}

\begin{remark}
Zorich maps may also be defined in $\RR^n$ for $n\geq 4$;
see~\cite{Martio75}. While it seems that the methods of
this paper extend to this more general case, we have restricted to
the case $n=3$ for simplicity.  We note that
Iwaniec and Martin~\cite{Iwaniec01}, whose presentation we have
followed in the definition of Zorich maps,
also confine themselves to the case $n=3$. Restriction to this case thus
allows to use their results directly.
\end{remark}

\section{Preliminaries}\label{prelims}
We suppress the index $a$ and 
write $f=(f_1,f_2,f_3)$
instead of~$f_a$. For $r=(r_1,r_2)\in \ZZ^2$ we put
\[
P(r)=P(r_1,r_2):=
\left\{(x_1,x_2)\in \RR^2:\;|x_1-2r_1|<1,\;|x_2-2r_2|<1\right\}
\] 
so that $P(0,0)$ is the interior of~$Q$. For $c\in \RR$ we define the half-space
\[
H_{> c}:=\left\{(x_1,x_2,x_3)\in \RR^3:\; x_3> c\right\}.
\]
The half-spaces $H_{<c}$, $H_{\geq c}$ and $H_{\leq c}$
and the plane  $H_{=c}$
are defined
analogously. Now $F$ maps $P(r_1,r_2)\times \RR$ bijectively onto
$H_{>0}$ if $r_1+r_2$ is even and bijectively onto $H_{< 0}$ if
$r_1+r_2$ is odd. Thus $f$ maps $P(r_1,r_2)\times \RR$ bijectively
onto $H_{>-a}$ or $H_{<-a}$, depending on whether $r_1+r_2$
is even or odd. For 
\[
r=(r_1,r_2)\in
S:=\left\{(s_1,s_2)\in \ZZ^2:\; s_1+s_2\ \text{even}\right\}
\] 
we define
\[
T(r):=P(r)\times (M,\infty).
\]
Since
$f(P(r)\times \RR)=H_{> -a}$  for $r\in S$ and
\begin{equation}\label{2a1}
f_3(x_1,x_2,x_3)= e^{x_3}h_3(x_1,x_2) -a\leq e^{x_3}-e^M+m\leq m<M\quad \text{for } x_3\leq M
\end{equation}
and hence $f(P(r)\times (-\infty,M])\subset H_{<M}$ we
see that $f(T(r))\supset H_{\geq M}$.
Thus there exists a branch $\Lambda^r:H_{\geq M}\to T(r)$ of the
inverse function of~$f$.
With
$\Lambda:=\Lambda^{(0,0)}$
we have
\[
\Lambda^{(r_1,r_2)}(x)=\Lambda(x)+(2r_1,2r_2,0)
\]
for all $x\in H_{\geq M}$ and all $r\in S$.
 
We shall need
some estimates for the derivative~$D\Lambda^r$. Since
$D\Lambda^r(x)=D\Lambda(x)$ whenever these derivatives exist
it 
suffices to obtain these estimates for~$D\Lambda$.
First we note that
\begin{equation}\label{2a}D\Lambda(x)=Df(\Lambda(x))^{-1}
\end{equation}
for $x\in H_{\geq M}$. Since $\Lambda(x)\in T(0,0)\subset H_{\geq
M}$ for $x\in H_{\geq M}$ and since $DF=Df$ we deduce from~\eqref{1e} that
$|D\Lambda(x)|\leq\alpha$ a.e. for $x\in H_{\geq M}$. This
implies that
\begin{equation}\label{2b}
\left|\Lambda(x)-\Lambda(y)\right|\leq|x-y|
\esssup\limits_{z\in [x,y]}\left|D\Lambda(z)\right|\leq \alpha |x-y|
\quad\text{for } x,y\in H_{\geq M}.
\end{equation}
 Next we note that~\eqref{1c} implies that
there exist positive constants $c_1$ and~$c_2$ such that
\begin{equation}\label{2c}c_1e^{x_3}\leq\ell
\left(Df(x_1,x_2,x_3)\right)\leq|Df(x_1,x_2,x_3)|\leq c_2e^{x_3}
\quad\text{a.e.} 
\end{equation}
Noting that
\begin{equation}\label{2c1}
|f(y_1,y_2,y_3)|-a \leq 
|F(y_1,y_2,y_3)|=e^{y_3}\leq |f(y_1,y_2,y_3)|+a
\end{equation}
for all $(y_1,y_2,y_3)\in \RR^3$ 
we deduce from~\eqref{2a} and~\eqref{2c} that
\[
\ell\left(D\Lambda(x)\right)\geq 
\frac{1}{\left|DF(\Lambda(x))\right|}\geq \frac{1}{c_2 \exp
\left(\Lambda_3(x)\right)}
\geq
\frac{1}{c_2\left(\left|f\left(\Lambda(x)\right)\right|+a\right)}=\frac{1}{c_2(|x|+a)}
\quad\text{a.e.} 
\]
Thus there exists $c_3>0$ such that
\begin{equation}\label{2e}\ell\left(D\Lambda(x)\right)\geq
\frac{c_3}{|x|}
\quad\text{a.e.} 
\end{equation} 
for $x\in H_{\geq M}$. Similarly
we have
\begin{equation}\label{2f}\left|D\Lambda(x)\right|\leq\frac{c_4}{|x|}
\quad\text{a.e.} 
\end{equation}
for some constant $c_4>0$, as well as
\begin{equation}\label{2g}\frac{c_5}{|x|^3}\leq
J_{\Lambda}(x)\leq\frac{c_6}{|x|^3}
\quad\text{a.e.} 
\end{equation} 
where
$c_5,c_6>0$.

Let now $x,y \in H_{\geq M}$. Then $x$ and $y$ can be connected by
a path $\gamma$ in
\[
H_{\geq M} \cap \left\{z\in\RR^3:\; |z|\geq \min \{|x|,|y|\}\right\}
\] 
whose length is not greater than $\pi|x-y|$. Together with~\eqref{2f}
this yields
\begin{equation}\label{2h}\left|\Lambda(x)-\Lambda(y)\right|\leq\pi|x-y|
\esssup\limits_{z\in \gamma}\left|D\Lambda(z)\right|
\leq c_4\pi\frac{|x-y|}{\min\{|x|,|y|\}}
\end{equation}

Next we note that there exists a unique point $(v_1,v_2)\in Q$ such
that $h(v_1,v_2)=(0,0,1)$. Then $F(v_1,v_2,x_3)=(0,0,e^{x_3})$ and
hence $f(v_1,v_2,x_3)=(0,0,e^{x_3}-a)$. It follows that if
$r=(r_1,r_2)\in S$, then
\[
\Lambda^r(0,0,e^{x_3}-a)=(v_1+2r_1, v_2+2r_2,x_3)
\] 
for $x_3\in \RR$ and thus
\begin{equation}\label{2j}\Lambda^r(0,0,y)=(v_1+2r_1,v_2+2r_2,\log(y+a))
\end{equation}
for $y\geq M$.

As in~\cite{Schleicher03} we will consider the function
\[
E:[0,\infty)\to[0,\infty),\;E(t)=e^t-1.
\]
We have $E(0)=0$ while $\lim_{k\to\infty}E^k(t)=\infty$ if $t>0$.
Later we will use that if $b>1$, then
\[
\log\left(E^{k+1}(t)+b\right)
=\log\left(\exp(E^k(t))-1+b\right)
=E^k(t)+\log\left(1+\frac{b-1}{\exp(E^k(t))}\right)
\]
so that
\begin{equation}\label{2k}
\log\left(E^{k+1}(t)+b\right)=E^k(t)+R_k(t)
\quad\text{with}\ 0\leq R_k(t)\leq \log b.
\end{equation}
We also note that if $0<t'<t''<\infty$, then
\begin{equation}\label{2l}
\lim_{k\to\infty}\left(E^k(t'')-E^k(t')\right)
=\infty \quad \text{ and }\quad
\lim_{k\to\infty}\frac{E^k(t'')}{E^k(t')}=\infty.
\end{equation}

\section{Construction of the hairs}\label{constructionhairs}
In this section we prove that the fixed point $\xi$ as given in
the theorem exists and that~$J$ is a union of uncountably many
pairwise disjoint hairs. We thus prove the part of the conclusion
of Theorem~\ref{thm1} that is the analogue of Theorem~A.

We first note that $f(H_{\leq M})\subset H_{\leq m}$
by~\eqref{2a1}. 
Moreover,
\[
|f(x)-f(y)|\leq |x-y|
\esssup\limits_{z\in [x,y]}|Df(z)|\leq
\alpha|x-y|
\] 
for $x,y\in H_{\leq m}$ by~\eqref{1d}. Banach's
fixed point theorem now implies that there exists a unique fixed
point $\xi\in H_{\leq m}$ and that $f^n(x)\to \xi$ as $n\to
\infty$ for all $x\in H_{\leq M}$.

If $r_1+r_2$ is odd, then $f(P(r_1,r_2)\times \RR)= H_{< -a}$.
Since $-a<m<M$ we deduce that $f$ maps the closure of
$P(r_1,r_2)\times \RR$ into $H_{\leq M}$ if $r_1+r_2$ is odd.
Thus 
\[
J\subset \bigcup_{r\in S} T(r).
\]
As in the case of exponential maps we associate to each $x\in J$ a
sequence 
\[
\underline{s}(x)=s_0s_1s_2\ldots =(s_k)_{k\geq 0}
\] 
in~$S$, where $s_k=(s_{k,1},s_{k,2})\in S$ is chosen such that $f^k(x)\in
T(s_k)$ for all $k\geq 0$. The sequence $\underline{s}(x)$ is
called the \emph{itinerary} (or \emph{external address}) of~$x$.
We denote the set of all sequences $\underline{s}:\NN\cup\{0\}\to S$ by
$\Sigma$. We say that $\underline{s}=(s_k)_{k\geq 0}\in\Sigma$ is
\emph{admissible} (or \emph{exponentially bounded}) if there exists
$t>0$ such that
\begin{equation}\label{3a}\limsup_{k\to \infty}\frac{|s_k|}{E^k(t)}<\infty.
\end{equation}
Here $|s_k|=|(s_{k,1},s_{k,2})|=\sqrt{s^2_{k,1}+s^2_{k,2}}$ is the
euclidean norm of $|s_k|$.

With these notations we have the following two propositions.

\begin{proposition} \label{prop1}
Let $x\in J$. Then $\underline{s}(x)$ is admissible.
\end{proposition}

\begin{proposition} \label{prop2}
Let $\underline{s}\in \Sigma$ be admissible. Then
$\{x\in J:\underline{s}(x)=\underline{s}\}$ is a hair.
\end{proposition}

Since $\{x\in J:\underline{s}(x)=\underline{s}\}\cap \{x\in
J:\underline{s}(x)=\underline{s}'\}=\emptyset$ for
$\underline{s},\underline{s}'\in
\Sigma$, $\underline{s}\neq\underline{s}'$, it follows from these two
propositions that $J$ is a union of pairwise disjoint hairs.
Moreover, the set of admissible sequences -- and thus the set of
hairs -- is easily seen to be uncountable.

The proof of Proposition~\ref{prop1} is straightforward, using exactly the
same reasoning as in the case of exponential maps. Therefore we
omit it here.

For the proof of Proposition~\ref{prop2} we suitably modify the arguments
of Schleicher and Zimmer~\cite{Schleicher03}. We fix an admissible
sequence $\underline{s}$ and
denote by
$t_{\underline{s}}$ be the infimum of all $t>0$ for which~\eqref{3a}
holds. It follows from~\eqref{2l} that
\begin{equation}\label{3j}\limsup_{k\to\infty}\frac{|s_k|}{E^k(t)}=\infty \quad
\ \text{for}\ 0<t<t_{\underline{s}}
\end{equation}
and
\begin{equation}\label{3k}\lim_{k\to\infty}\frac{|s_k|}{E^k(t)}=0\quad
\ \text{for}\ t>t_{\underline{s}}.
\end{equation}
Choosing $t_k\in [0,\infty)$ such that $2|s_k|=E^k(t_k)$ 
and putting $\tau_k:=\sup_{j\geq k} t_j$ we have
\[
t_{\underline{s}} =\limsup_{k\to\infty} t_k
=\lim_{k\to\infty} \tau_k.
\]
We use the abbreviation
\[
L_k:=\Lambda^{s_k}=\Lambda^{(s_{k,1},s_{k,2})}.
\]
For $k\geq 0$ we define
\[
g_k:[0 ,\infty)\to H_{\geq M},\ 
g_k(t)=(L_0\circ L_1\circ \ldots\circ L_k)(0,0,E^{k+1}(t)+M).
\]

\begin{lemma} \label{lemma1}
The sequence $(g_k)$ converges locally uniformly
on $(t_{\underline{s}},\infty)$.
\end{lemma}
\begin{proof}It follows from~\eqref{2j} and~\eqref{2k} that
\begin{equation}\label{3b}
\begin{aligned}
L_k(0,0,E^{k+1}(t)+M)
&=
(v_1+2s_{k,1},v_2+2s_{k,2},\log(E^{k+1}(t)+M+a))\\
&=
(v_1+2s_{k,1},v_2+2s_{k,2},E^k(t)+R_k(t))
\end{aligned}
\end{equation}
where $0\leq R_k(t)\leq \log(M+a)$. Since $|v_1|\leq 1$ and
$|v_2|\leq 1$ this yields
\[
\left|L_k(0,0,E^{k+1}(t)+M)\right|\geq
\left|(2s_{k,1},2s_{k,2},E^k(t))\right|-c_7
\]
with
$c_7:=2+\log(M+a)$.
Since $L_k(H_{\geq M})\subset H_{\geq M}$ this implies that
\begin{equation}\label{3d2}\left|L_k(0,0,E^{k+1}(t)+M)\right|
\geq  \max\{2|s_k|-c_7,E^k(t)-c_7,M\}
\end{equation}
Also, it follows from~\eqref{3b} that
\begin{equation}\label{3e}
\left|L_k(0,0,E^{k+1}(t)+M)-(0,0,E^k(t)+M)\right|
\leq 2|s_k|+c_7+M .
\end{equation}
Applying~\eqref{2h} with $x=L_k(0,0,E^{k+1}(t)+M)$,
$y=(0,0,E^k(t)+M)$ and $\Lambda=L_{k-1}$ 
and noting that $E^k(t)\geq 2|s_k|$ for $t\geq t_k$ we deduce
from~\eqref{3d2} and~\eqref{3e} that if $t\geq t_k$, then
\[
\left|L_{k-1}(L_k(0,0,E^{k+1}(t)+M))-L_{k-1}(0,0,E^k(t)+M)\right|
\leq
c_4 \pi\frac{E^k(t)+c_7+M}{\max\{E^k(t)-c_7, M\}}
\leq c_8
\]
for some constant~$c_8$.
Now~\eqref{2b} gives
\[
\left|L_{k-2}(L_{k-1}(L_k(0,0,E^{k+1}(t)+M)))-L_{k-2}(L_{k-1}(0,0,E^k(t)))\right|\leq
\alpha c_8
\] 
and induction yields
\begin{equation}\label{3f}
\left|g_k(t)-g_{k-1}(t)\right|\leq \alpha^{k-1}c_8
\quad \text{for}\ t\geq t_k.
\end{equation}
In particular, $\left|g_k(t)-g_{k-1}(t)\right|\leq \alpha^{k-1}c_8$
if $k\geq j$ and $t\geq \tau_j$.
As $\lim_{j\to\infty}\tau_j= t_{\underline{s}}$
this implies that $(g_k)$ converges locally uniformly
on $\left(t_{\underline{s}},\infty\right)$.
\end{proof}

Define $g:\left(t_{\underline{s}},\infty\right)\to H_{\geq M}$ by
$g(t)=\lim_{k\to\infty} g_k(t)$.
Then $g$ is continuous and~\eqref{3f} yields
\begin{equation}\label{3f1}
\left|g(t)-g_{k-1}(t)\right|
\leq
\frac{\alpha^{k-1}c_8}{1-\alpha} \quad \text{for}\ t\geq \tau_k.
\end{equation}
We also have
\begin{equation}\label{3f3}
\left|g_l(t)-g_{k-1}(t)\right|
\leq
\frac{\alpha^{k-1}c_8}{1-\alpha} \quad \text{for}\ t\geq \tau_k
\ \text{and}\ l\geq k.
\end{equation}
In particular,~\eqref{3f1} yields
\begin{equation}\label{3f2}
|g(t)-g_0(t)|\leq \frac{c_8}{1-\alpha}
\quad \text{for}\ t\geq \tau_1.
\end{equation}
Since
\[
g_0(t)=L_0(0,0,E(t)+M)=(v_1+2s_{0,1},v_2+2s_{0,2},\log
(E(t)+M+a))
\] 
by~\eqref{2j} we deduce from~\eqref{2k} that
\[
\left|g_0(t)-(2s_{0,1},2s_{0,2},t)\right|\leq 2+\log (M+a)=c_7.
\]
Combining this with~\eqref{3f2} we obtain
\begin{equation}\label{3g}
\left|g(t)-(2s_{0,1},2s_{0,2},t)\right|\leq c_9
\quad \text{for}\ t\geq \tau_1,
\end{equation}
where $c_9:=c_7+c_8/(1-\alpha)$.

\begin{lemma}\label{lemma3}
The sequence $(g_k)$ has a subsequence which converges uniformly
on $[t_{\underline{s}},\infty)$ and thus
$g$ extends to a continuous map
$g:[t_{\underline{s}},\infty)\to H_{\geq M}$.
\end{lemma}
\begin{proof}
The conclusion follows immediately from~\eqref{3f}
if $t_k\leq t_{\underline{s}}$ for all large $k$ and
thus $\tau_k=t_{\underline{s}}$ for large~$k$, since then
$(g_k)$ converges uniformly on $[ t_{\underline{s}},\infty)$.
We may thus assume that there exist arbitrarily large $k$
for which  $t_k>t_{\underline{s}}$. Then
$\tau_k>t_{\underline{s}}$ for all $k$ and
there exists
an increasing sequence $(k_j)$ such that
$\tau_k=t_{k_j}$ for $k_{j-1}<k\leq k_j$.
Thus $\tau_{k_j+1}=\tau_{k_{j+1}}$ and~\eqref{3f3} implies that
\begin{equation}\label{3f8}
\left|
g_{k_{j+1}}\left(\tau_{k_{j+1}}\right)-g_{k_{j}}\left(\tau_{k_{j+1}}\right)
\right|
\leq \frac{\alpha^{k_{j}}c_8}{1-\alpha}
\leq \frac{\alpha^{j}c_8}{1-\alpha}.
\end{equation}
It follows from~\eqref{3b}
that if $0\leq t\leq t_k$, then
\begin{eqnarray*}
  &  & \left|L_k(0,0,E^{k+1}(t_k)+M)-L_k(0,0,E^{k+1}(t)+M)\right| \\
  & = & \left|\left(0,0, E^k(t_k)-E^k(t)+R_k(t_k)-R_k(t)\right)\right| \\
  & \leq & E_k(t_k)+2\log(M+a) \\
  & = & 2|s_k|+2\log(M+a).
\end{eqnarray*}
Combining 
this with~\eqref{2h} and~\eqref{3d2} we find that
\begin{eqnarray*}
& &
\left|L_{k-1}(L_k(0,0,E^{k+1}(t_k)+M))-L_{k-1}(L_k(0,0,E^{k+1}(t)+M))\right|\\
& \leq &
c_4\pi \frac{2|s_k|+2\log(M+a)}{\max\{2|s_k|-c_7,M\}}
\\
&\leq &
c_{10}
\quad\text{for}\ 0 \leq t\leq t_k
\end{eqnarray*}
with a constant $c_{10}$ 
and thus~\eqref{2b}
yields
\[
\left|g_k(t_k)-g_k(t)\right|\leq \alpha^{k-1}c_{10}
\quad\text{for}\ 0 \leq t\leq t_k.
\]
Hence
\[
\left|
g_{k_{j+1}}(\tau_{k_{j+1}})-g_{k_{j+1}}(t) \right|
\leq \alpha^{k_{j+1}-1}c_{10}
\leq \alpha^{j}c_{10}
\quad\text{for}\ 0 \leq t\leq \tau_{k_{j+1}}
\]
and, since $\tau_{k_{j+1}}\leq\tau_{k_{j}}$, also
\[
\left|
g_{k_{j}}(\tau_{k_{j}})-g_{k_{j}}(t) \right|
\leq \alpha^{j-1}c_{10}
\quad\text{for}\ 0 \leq t\leq \tau_{k_{j+1}}.
\]
In particular,
\[
\left|
g_{k_{j}}(\tau_{k_{j}})-g_{k_{j}}(\tau_{k_{j+1}}) \right|
\leq \alpha^{j-1}c_{10}.
\]
Combination of the last three inequalities with~\eqref{3f8} yields
\begin{eqnarray*}
\left|g_{k_{j+1}}(t) - g_{k_{j}}(t)\right|
&\leq&
\left|g_{k_{j+1}}(t)-g_{k_{j+1}}(\tau_{k_{j+1}})\right|
+\left|g_{k_{j+1}}(\tau_{k_{j+1}})-g_{k_{j}}(\tau_{k_{j+1}})\right|\\
& &
+\left|g_{k_{j}}(\tau_{k_{j+1}})-g_{k_{j}}(\tau_{k_{j}})\right|
+\left|g_{k_{j}}(\tau_{k_{j}})- g_{k_{j}}(t)\right|\\
&\leq&
\left( 3c_{10}+\frac{\alpha}{1-\alpha}c_8\right) \alpha^{j-1}
\quad\text{for}\ 0 \leq t\leq \tau_{k_{j+1}}.
\end{eqnarray*}
On the other hand,~\eqref{3f3} implies that
\[
\left|g_{k_{j+1}}(t) -g_{k_{j}}(t)\right|
\leq\alpha^{k_{j}-1} \frac{c_8}{1-\alpha}\leq \alpha^{j-1}
\frac{c_8}{1-\alpha}
\quad\text{for}\ t\geq \tau_{k_{j}+1}= \tau_{k_{j+1}}.
\]
The last two inequalities show that $\left(g_{k_{j}}\right)$
converges uniformly on $[0,\infty)$ and thus in particular
on $[t_{\underline{s}},\infty)$.
\end{proof}

We note that $g_k$ and $g$ depend on~$\underline{s}$. Using the
self-explanatory notation $g_{\underline{s},k}$ and
$g_{\underline{s}}$ we find that
\[
f(g_{\underline{s},k}(t))=g_{\sigma(\underline{s}),k-1}(E(t))
\]
where $\sigma:\Sigma\to
\Sigma$, $\sigma(s_0s_1s_2\ldots)=s_1s_2s_3\ldots$, is the shift map.
Thus
$f(g_{\underline{s}}(t)) =g_{\sigma(\underline{s})}(E(t))$
and
\begin{equation}\label{3i}f^k(g_{\underline{s}}(t))
=g_{\sigma^k(\underline{s})}(E^k(t))
\end{equation}
for $t\geq 0$ and $k\in \NN$.
Similarly, the sequences $(t_k)$ and $(\tau_k)$ depend on~$\underline{s}$.
Using the notation $t_{\underline{s},k}$ and $\tau_{\underline{s},k}$
a more precise formulation of~\eqref{3g} would thus be
\begin{equation}
\label{3g1}
\left|g_{\underline{s}}(t)-(2s_{0,1},2s_{0,2},t)\right|\leq c_9
\quad \text{for}\ t\geq \tau_{\underline{s},1}.
\end{equation}
We also note that
\[
E^{j+k}(t_{\underline{s},j+k})
=2|s_{j+k}|
=E^{j} (t_{\sigma^k(\underline{s}),j})
\]
and hence
\[
E^{k}(t_{\underline{s},j+k})
=t_{\sigma^k(\underline{s}),j}
\quad \text{and} \quad
E^{k}(\tau_{\underline{s},j+k})
=\tau_{\sigma^k(\underline{s}),j}
\]
for $j,k\geq 0$.

For  $x=x_0\in J$ and $k\geq 0$
we put $x_k=(x_{k,1}, x_{k,2}, x_{k,3}):=f^k(x)$.
Similarly we write
$y_k=(y_{k,1}, y_{k,2}, y_{k,3}):=f^k(y)$.

\begin{lemma}\label{lemma2a}
There exists a positive constant~$H$ such that if $x,y\in J$ with
$\underline{s}(x)=\underline{s}(y)$ and if
$y_{k,3}\geq  x_{k,3} +H$ for some $k\in\NN$, then
$y_{j,3}\geq  x_{j,3} +H$ for all $j>k$.
\end{lemma}
\begin{proof}
Since $\underline{s}(x)=\underline{s}(y)$ we have
$\left|\left( y_{k,1}, y_{k,2}\right)-
\left( x_{k,1}, x_{k,2}\right)\right|\leq 4$.
Thus
\[
y_{k+1,3}
\geq 
\left|y_{k+1}\right|-\left|\left( y_{k,1}, y_{k,2}\right)\right|
\geq 
e^{y_{k,3}}-a -\left|\left( x_{k,1}, x_{k,2}\right)\right|-4
\geq 
e^He^{x_{k,3}}-a-\left|x_{k+1}\right|-4
\]
by~\eqref{2c1}.
If $\left|x_{k+1}\right|\geq 2a$, then $\exp(x_{k,3})\geq \left|x_{k+1}\right|-a\geq \frac12
\left|x_{k+1}\right|$ by~\eqref{2c1} and thus 
\[
y_{k+1,3}\geq \left(\tfrac12 e^H-1\right) \left|x_{k+1}\right|-a-4.
\]
If $\left|x_{k+1}\right|< 2a$, then 
\[
y_{k+1,3}\geq e^H e^M-3a-4\geq e^H e^M\frac{\left|x_{k+1}\right|}{2a} -3a-4.
\]
In both cases we 
we obtain $y_{k+1,3}>\left|x_{k+1}\right|+H\geq x_{k+1,3}+H$ if $H$ is
chosen sufficiently large. The conclusion follows by induction.
\end{proof}

\begin{lemma}\label{lemma2}
$g$ is injective on  $[t_{\underline{s}},\infty)$.
\end{lemma}

\begin{proof}
Let $t_{\underline{s}}<v<w$ and  and $k\in\NN$. Put
\[
y_k:=f^k\left(g_{\underline{s}}(v)\right)
=g_{\sigma^k(\underline{s})}\left(E^k(v)\right)
\quad\text{and}\quad
z_k:=f^k\left(g_{\underline{s}}(w)\right)
=g_{\sigma^k(\underline{s})}\left(E^k(w)\right).
\]
For large $k$ we have $v\geq \tau_{\underline{s},k+1}$ and
thus $E^k(w)\geq E^k(v)\geq E^k(\tau_{\underline{s},k+1})
=\tau_{\sigma^k(\underline{s}),1}$.
It follows from~\eqref{3g1} that
\[
\left|y_k
-(2s_{k,1},2s_{k,2},E^k(v))\right|\leq c_9
\quad\text{and}\quad
\left|z_k
-(2s_{k,1},2s_{k,2},E^k(w))\right|\leq c_9
\]
for such~$k$.
In particular, $\left|y_{k,3}-E^k(v)\right|\leq c_9$ and
$\left|z_{k,3}-E^k(w)\right|\leq c_9$ for large~$k$. From~\eqref{2l}
we deduce that $E^k(w)>E^k(v)+H+2c_9$ and thus
\begin{equation}\label{xk3}
z_{k,3}>y_{k,3} +H
\end{equation}
for large~$k$.
In particular, $z_k\neq y_k$ and thus $g_{\underline{s}}(w)
\neq g_{\underline{s}}(v)$.
Hence $g$ is injective on  $(t_{\underline{s}},\infty)$.

Lemma~\ref{lemma2a} and~\eqref{xk3} yield
that $ y_{k,3}<z_{k,3}+H$ for all $k\in\NN$.
The same argument shows that if $t_{\underline{s}}<u<v$ and
$x_k:=f^k(g_{\underline{s}}(u))$, then
$ x_{k,3}<y_{k,3}+H$ for all $k\in\NN$.
Let $x'_k:=f^k(g_{\underline{s}}(t_{\underline{s}}))$.
With $u\to t_{\underline{s}}$ we obtain $x'_{k,3}\leq y_{k,3}+H$
for all $k\in\NN$.
Using~\eqref{xk3} we obtain
$x'_{k,3}<z_{k,3}$
and thus
$g_{\underline{s}}(t_{\underline{s}})\neq g_{\underline{s}}(w)$.
Hence $g$ is in fact injective on $[t_{\underline{s}},\infty)$.
\end{proof}

\begin{lemma}\label{lemma7}
Let $x\in J$. Then
$x_{k,3}\geq E^k(t_{\underline{s}(x)})$
for $k\geq 0$.
\end{lemma}
\begin{proof}
We put
$\underline{s}=\left(s_k\right):=\underline{s}(x)$.
As the conclusion is trivial for $t_{\underline{s}}=0$ we
may assume that $t_{\underline{s}}>0$.
Since $a\geq e^M-M\geq 1$ we have
$x_{k+1,3}\leq \exp\left(x_{k,3}\right)-a\leq
\exp\left(x_{k,3}\right)-1=E\left(x_{k,3}\right)$
and thus
\begin{equation}\label{newc}
x_{n,3}\leq  E^{n-k} \left(x_{k,3}\right)
\end{equation}
for $n\geq k$.
Because $|x_{k,1}-2 s_{k,1}|\leq 1$ and
$|x_{k,2}-2 s_{k,2}|\leq 1$ we deduce from~\eqref{2c1} that
\begin{equation}\label{newa}
2|s_{k}|\leq |x_{k}|+2\leq \exp(x_{k-1,3}) +a +2
=E(x_{k-1,3})+a+3
\end{equation}
for all $k\geq 0$.
Let now $\delta\in (0,t_{\underline{s}})$.
Then there exists
arbitrarily large $l$ with
\begin{equation}\label{newb}
2|s_{l}|-a-3\geq E^l(t_{\underline{s}}-\delta).
\end{equation}
In particular, given $k\geq 0$ there exists $l> k$
with this property.
Combining~\eqref{newa} and~\eqref{newb}
we see that
$E^l(t_{\underline{s}}-\delta)\leq E(x_{l-1,3})$.
Hence
\[
E^{l-k} \left( E^k(t_{\underline{s}}-\delta) \right)
=
E^l(t_{\underline{s}}-\delta)\leq E(x_{l-1,3})
\leq  E^{l-k} \left(x_{k,3}\right)
\]
by~\eqref{newc} and thus
\[
E^k(t_{\underline{s}}-\delta)
\leq
x_{k,3}.
\]
As $\delta$ can be chosen arbitrarily small, the conclusion follows.
\end{proof}

\begin{lemma}\label{lemma4}
Let $x\in J$. Then there exist $t\geq
t_{\underline{s}(x)}$ with $x=g_{\underline{s}(x)}(t)$.
\end{lemma}

\begin{proof}
We again put
$\underline{s}=\left(s_k\right):=\underline{s}(x)$ and
note that
\[
L_k(f^{k+1}(x))=f^k(x)=(x_{k,1},x_{k,2},x_{k,3})
\]
where $\left|x_{k,1}-2s_{k,1}\right|\leq 1$,
$\left|x_{k,2}-2s_{k,2}\right|\leq 1$ and
$x_{k,3}\geq M$. Define $u_k$ by $x_{k,3}=E^k(u_k)$.
Lemma~\ref{lemma7} implies that $u_k\geq t_{\underline{s}}$.
Moreover, it follows from~\eqref{newc} that $(u_k)$ is
decreasing and hence convergent.

Using~\eqref{3b} we find that
\begin{eqnarray} \label{newd}
 & &   
\left|L_k(f^{k+1}(x))-L_k((0,0,M+E^{k+1}(u_k)))\right| \nonumber\\
  & = & \left|(x_{k,1}-2s_{k,1}-v_1,x_{k,2}-2s_{k,2}-v_2,
  -R_k(u_k))\right| \\
  & \leq & 4+\log(M+a). \nonumber
\end{eqnarray}
Since
$x=(L_0\circ L_1\circ \ldots \circ L_k)(f^{k+1}(x))$
we deduce similarly as in the proofs of Lemma~\ref{lemma1}
and Lemma~\ref{lemma2} from~\eqref{2b} and~\eqref{newd}  that
\[
|x-g_k(u_k)|\leq \alpha^{k}(4+\log(M+a))
\]
so that $x=\lim_{k\to\infty} g_k(u_k)$.
With $t:=\lim_{k\to\infty} u_k$ we
obtain $x=g_{\underline{s}}(t)$. 
\end{proof}

Proposition~\ref{prop2}
follows immediately from Lemmas~\ref{lemma1}--\ref{lemma4}. As 
already mentioned it yields together with 
Proposition~\ref{prop1} that $J$ is a union of pairwise
disjoint Devaney hairs. In fact, 
if $\Sigma'$ denotes the set of admissible sequences we have
\[
J=\bigcup_{{\underline s} \in\Sigma'}
g_{\underline s}\left(\left[ t_{\underline s},\infty\right)\right)
\quad\text{and}\quad
C=\left\{g_{\underline s}\left(t_{\underline s}\right):\; 
{\underline s} \in\Sigma'\right\}.
\]

\section{The Hausdorff dimension of the hairs} \label{hausdorffhairs}
In this section we prove that $\dim J =3$, thereby establishing
the analogue of Theorem~B.
Following McMullen~\cite{McMullen87} we consider for $k\in \NN$ a
finite collection $A_k$ of disjoint compact subsets of $\RR^n$
such that the following two conditions are satisfied:
\begin{enumerate}\item[(a)] every element of $A_{k+1}$ is contained in a
unique element of $A_k$;

\item[(b)] every element of $A_k$ contains at least one element of
$A_{k+1}$.
\end{enumerate}
Denote by $\overline{A}_k$ the union of all elements of $A_k$ and
put 
\[
A:=\bigcap^\infty_{k=1} \overline{A}_k. 
\]
Suppose that $(\Delta_k)$ and
$(d_k)$ are sequences of positive real numbers such that if $B\in
A_k$, then
\[
\dens\left(\overline{A}_{k+1},B\right)
:=\frac{\area\left(\overline{A}_{k+1}\cap B\right)}{\area(B)}\geq\Delta_k
\] 
and
\[
\diam B\leq d_k.
\]
Then we have the following result~\cite[Proposition~2.2]{McMullen87}.
\begin{lemma} \label{lemmamcm}
Let $A$, $A_k$, $\Delta_k$ and $d_k$ be as above. Then
\[
\limsup_{k\to \infty}\frac{\sum^{k+1}_{j=1}\;|\log\Delta _j|}{|\log d_k|}
\geq n-\dim A.
\]
\end{lemma}
The construction of the sets $A_k$ will be very similar to that 
in~\cite{McMullen87}, but in contrast to~\cite{McMullen87} we will
not have $\Delta_k\geq\Delta$ for some $\Delta>0$ and all~$k$, but
$\lim_{k\to \infty}\Delta_k=0$. However, the sequence
$(\Delta_k)$ will tend to $0$ much more 
slowly than $(d_k)$ so that Lemma~\ref{lemmamcm} can 
still be applied.

To begin the construction we note that there exists $q\in (0,1)$
such that
\[
\left\{(x_1,x_2,x_3)\in\RR^3:\; 
x_1^2+x_2^2+x_3^2=1, x_3\geq\tfrac{1}{2}\right\}
\subset
h\left(\left\{(x_1,x_2)\in \RR^2:\; 
|x_1|\leq q,|x_2|\leq q\right\}\right).
\]
For $r=(r_1,r_2)\in S$ and $\ell\in\NN$ we consider the box
\[
R(r,\ell)=R(r_1,r_2,\ell)=\left\{x\in \RR^3:\; |x_1-2r_1|\leq
q,\;|x_2-2r_2|\leq q,\;\ell \leq x_3 \leq \ell+\tfrac{3}{4}\right\}
\] 
and we denote by $U$ the collection of all $R(r,\ell)$; that is,
\[
U=\left\{R(r,\ell):\; r\in S,\;\ell\in\NN\right\}.
\]
We note that $F(R(r,\ell))$ does not depend on $r\in S$ 
and that
\[
\left\{x\in\RR^3:\; e^\ell\leq |x|\leq e^{3/4}e^\ell,\; x_3\geq
\tfrac{1}{2}|x|\right\}
\subset F(R(r,\ell))
\subset
\left\{x\in\RR^3:\; e^\ell\leq
|x|\leq e^{3/4}e^\ell\right\}.
\]
Since $2<\exp\left(\frac{3}{4}\right)<3$
this yields 
\begin{equation}\label{4a}
\left\{x\in\RR^3:\;e^\ell\leq|x|\leq 2e^\ell,\;x_3
\geq\tfrac{1}{2}|x|\right\} 
\subset  f(R(r,\ell)) 
\subset \left\{x\in \RR^3:
\;\tfrac{1}{2}e^\ell\leq|x|\leq 3e^\ell\right\}
\end{equation}
if $\ell$ is large. We put
\[
U(\ell):=\{R\in U:\; R\subset f(R(r,\ell))\}
\quad\text{and}\quad
\overline{U}(\ell):=\bigcup_{R\in U(\ell)}R,
\]
noting that this definition does not depend on~$r\in S$.
We deduce from~\eqref{4a} that there exists a positive constant $\delta$ 
such that
\begin{equation}\label{4b} 
\dens\left(\overline{U}(\ell),f(R(r,\ell))\right)\geq \delta
\end{equation}
for all sufficiently large~$\ell$. 

Let now $B_0:=R(0,0,\ell_0)$
where $\ell_0$ is so large that~\eqref{4a} and~\eqref{4b} hold for
$\ell\geq\ell_0$. We put $A_0:=\{B_0\}$.

We will now construct the sets $A_k$ inductively such that if
$B\in A_k$, then there exist $r_j=(r_{j,1},r_{j,2})\in S$ and
$\ell_j\in \NN$, for $j=1,2,\ldots,k$, such that
\[
f^k(B)=R(r_k,\ell_k),
\]
\[
f^j(B)\subset R(r_j,\ell_j)\quad \text{for}\ j=1,2,\ldots,k-1
\]
and
\[
\ell_j\geq \tfrac{1}{2}\exp (\ell_{j-1})\geq \ell_{j-1}\geq
\ell_0
\quad \text{for}\ j=1,2,\ldots,k.
\]
Assuming that $A_k$ has been constructed
and $B$ is as above, we put
\[
A_{k+1}(B):=\left\{(L_0\circ L_1\circ \ldots \circ L_k)(R):\; 
R\in U(\ell_k)\right\}
\] 
and
\[
A_{k+1}:=\bigcup_{B\in A_k} A_{k+1}(B),
\]
with $L_j=\Lambda^{r_j}$ as in section~\ref{constructionhairs}.
Then $A_{k+1}$ has the
required properties.

It follows from~\eqref{2g}, \eqref{4a} and~\eqref{4b} that if $B\in A_k$ 
with $r_j$ and $\ell_j$ as above, then
\[
\dens\left(L_k(\overline{U}(\ell_k)),R(r_k,\ell_k)\right)
=\dens\left(L_k(\overline{U}(\ell_k)),L_k(f(R(r_k,\ell_k)))\right)
\geq
\eta \delta
\] 
where $\eta:= c_5/(216 c_6)$.  
Since
\[
(L_j\circ\ldots\circ L_k)(f(R(r_k,\ell_k)))\subset R(r_j,\ell_j)\subset
\left\{x\in\RR^3:\; \tfrac12 e^{l_{j-1}}\leq |x|\leq 3 e^{l_{j-1}}\right\}
\]
for $1\leq j\leq k-1$
we conclude, using again~\eqref{2g},  
by induction that
\begin{equation}\label{4c}
\begin{aligned}
\dens\left(\overline{A}_{k+1},B\right)
&
= 
\dens\left((L_0\circ \ldots\circ L_k)\left(\overline{U}(\ell_k)\right),
    (L_0\circ\ldots\circ L_k)(f(R(r_k,\ell_k)))\right) 
\\ &
\geq  \eta^{k+1}\delta.
\end{aligned}
\end{equation}
On the other hand, since $R(r_k,\ell_k)\subset H_{\geq \ell_k}$ we deduce from~\eqref{2h} that
\[
\diam L_{k-1}(R(r_k,\ell_k))\leq\frac{c_4\pi}{\ell_{k}}\diam R(r_k,\ell_k)\leq
\frac{3 c_4\pi}{\ell_{k}}.
\] 
Putting
$E_*(t)=\frac{1}{2}e^t$ 
and noting that
$\ell_j\geq E_*(\ell_{j-1})$ we find that
\[
\diam L_{k-1}(R(r_k,\ell_k))\leq
\frac{3 c_4\pi}{E_*^k(\ell_0)}.
\] 
Using~\eqref{2b} we deduce that
\begin{equation} \label{4d}
 \diam B= \diam \left((L_0\circ\ldots\circ L_{k-1})(R(r_k,\ell_k))\right) 
\leq \alpha^{k-1} \frac{3 c_4\pi}{E_*^k(\ell_0)}
\leq\frac{1}{2E_*^k(\ell_0)} 
\end{equation}
for large~$k$.

Because of~\eqref{4c} and~\eqref{4d} we can apply Lemma~\ref{lemmamcm}
with
\[
\Delta_k:=\eta^{k+1}\delta \quad\text{and}\quad\
d_k:=\frac{1}{2 E_*^k(\ell_0)}.
\]
Since $\eta<1$ and $\delta<1$ we
have
\[
 \sum^{k+1}_{j=1} |\log \Delta_j|
 =\sum^{k+1}_{j=1}\left(\log\frac{1}{\delta}+   (j+1) \log\frac{1}{\eta}\right)
 =(k+1)\log\frac{1}{\delta}+\frac{k^2+5k+4}{2} \log\frac{1}{\eta}
 \leq  k^3
\]
for large~$k$. On the other hand,
\[
|\log d_k|=\log\left(2 E_*^{k}(\ell_0)\right)
= E_*^{k-1}(\ell_0)
\]
for large~$k$. It is not difficult to see that
\[
\lim_{k\to\infty}
\frac{k^3}{ E_*^{k-1}(\ell_0)}=0.
\] 
Hence $\dim A=3$
by Lemma~\ref{lemmamcm}.
Since clearly $A\subset J$ we also have $\dim J=3$.
\begin{rem}
It follows from the construction that $A\subset I(f)$. Thus
we also have $\dim I(f)=3$.
\end{rem}

\section{The Hausdorff dimension of the hairs without endpoints}
\label{hausdorffwithout}

In this section we prove that $J\backslash C$ has Hausdorff
dimension 1, thereby establishing the analogue of Theorem~C and
thus completing the proof of Theorem~\ref{thm1}.

The following lemma is standard~\cite[Lemma~4.8]{Falconer90}. 
Here the open
ball of radius $r$ around a point $x\in \RR^n$ is denoted by
$B(x,r)$.

\begin{lemma} \label{lemmafalc}
Let $K\subset \RR^n$ be bounded, $R>0$ and
$r:K\to(0,R]$. Then there exists an at most countable subset $L$
of $K$ such that $B(x,r(x))\cap B(y,r(y))=\emptyset$ for $x,y\in
L$, $x\neq y$, and
\[
\bigcup_{x\in K} B(x,r(x))\subset \bigcup_{x\in L}
B(x,4r(x))
\] 
\end{lemma}

We shall deduce the following result from Lemma~\ref{lemmafalc}.

\begin{lemma} \label{lemma5}
Let $K\subset \RR^n$ be bounded and let $\rho>1$.
Suppose that for all $x\in K$ and $\delta >0$ there exist
$r(x)\in (0,1)$, $d(x)\in (0,\delta)$ and $N(x)\in \NN$ satisfying
$N(x)d(x)^\rho\leq r(x)^n$ such that $B(x,r(x))\cap K$ can be
covered by $N(x)$ sets of diameter at most $d(x)$. Then
$\dim K\leq\rho$.
\end{lemma}

\begin{proof} Choose $R>0$ such that $K\subset B(0,R)$ and let
$\delta>0$. Since
\[
K\subset \bigcup_{x\in K} B\left(x,\tfrac{1}{4}r(x)\right),
\]
Lemma~\ref{lemmafalc} yields the existence of an at most countable subset $L$
of $K$ such that
\[
K\subset \bigcup_{x\in L} B(x,r(x))
\]
while
\[
B\left(x,\tfrac{1}{4}r(x)\right)\cap B\left(y,\tfrac{1}{4}r(y)\right)
=\emptyset\quad  \text{for}\ x,y\in L, \;x\neq y.
\] 
For each $x\in L$, let $A_1(x),A_2(x),\ldots,A_{N(x)}(x)$ be the
sets of diameter at most $d(x)$ which cover $B(x,r(x))\cap K$ so
that $N(x)d(x)^\rho\leq r(x)^n$. Then
\[
K\subset \bigcup_{x\in L}\bigcup _{j=1}^{N(x)}A_j(x).
\]
Now
\[
\sum_{x\in L}\sum^{N(x)}_{j=1}\left(\diam A_j(x)\right)^\rho
\leq
\sum_{x\in L} N(x) d(x)^\rho\leq \sum_{x\in L} r(x)^n.
\]
Since $r(x)\leq \delta$ we have
$B\left(x,\tfrac{1}{4}r(x)\right)\leq
B\left(0,R+\tfrac{1}{4}\delta\right)$ for all $x\in L$. Since the balls
$B\left(x,\tfrac{1}{4}r(x)\right)$, $x\in L$, are pairwise disjoint,
this yields
\[
\sum_{x\in L}\left(\tfrac{1}{4}r(x)\right)^n\leq
\left(R+\tfrac{1}{4}\delta\right)^n.
\]
We obtain
\[
\sum_{x\in L}\sum^{N(x)}_{j=1}
\left(\diam A_j(x)\right)^\rho
\leq(4R+\delta)^n.
\] 
Thus the $\rho$-dimensional
Hausdorff measure of $K$ is finite. In particular, $\dim K\leq
\rho$.
\end{proof}

As in Karpi\'nska's paper~\cite{Karpinska99} the key idea
is to show that points in
$J\backslash C$ escape to $\infty$ in a comparatively small
region. To define such a region $\Omega$, we consider the function
\[
\psi:[1,\infty)\to[1,\infty),\ 
\psi(x)=\exp \left(\sqrt{\log x}\right),
\]
and put
\[
\Omega := \left\{(x_1,x_2,x_3)\in
\RR^3:\; x_3>\max\{1,M\}\ \text{and}\ 
x_1^2+x_2^2<\psi(x_3)^2\right\}.
\]
It is not difficult to see that
\begin{equation}\label{5a}\lim_{x\to\infty} \frac{\psi(x)}{x^\varepsilon}=0
\quad\text{for}\ \varepsilon>0
\end{equation}
and that
\begin{equation}\label{5b}\lim_{x\to\infty} \frac{E^k(t')}{\psi(E^k(t))}=0
\quad\text{for}\ 0<t'<t.
\end{equation}

\begin{lemma} \label{lemma6}
If $x\in J\backslash C$, then $f^k(x)\in \Omega$ for
all large~$k$.
\end{lemma}

\begin{proof} Let $x=g_{\underline{s}}(t)$ where
$t>t_{\underline{s}}$. Since
$f^k(x)=g_{\sigma^k(\underline{s})}(E^k(t))$ by~\eqref{3i} we have
\begin{equation}\label{5c}
|f^k(x)-(2s_{k,1}, 2s_{k,2},E^k(t))|\leq c_9
\end{equation}
by~\eqref{3g1}. Let now $t'\in (t_{\underline{s}},t)$. Then
\[
\lim_{k\to\infty}\frac{|s_k|}{E^k(t')}=0
\] 
by~\eqref{3k}.
Combining this with~\eqref{5b} yields
\begin{equation}\label{5d}\lim_{k\to\infty}
\frac{|s_k|}{\psi\left(E^k(t)\right)}=0.
\end{equation}
The conclusion now follows from~\eqref{5c} and
\eqref{5d}.
\end{proof}

We define
\[
J':=\left\{x\in J\backslash C:\;f^k(x)\in \Omega\ \text{for all}\ k\geq
0\right\}.
\]
We shall show that
\begin{equation}\label{5e}\dim J'=1.
\end{equation}
Since
\[
J\backslash C=\bigcup_{k\geq 0} f^{-k}(J')
\]
by Lemma~\ref{lemma6}, since $f$ is locally bilipschitz, and since
bilipschitz maps preserve Hausdorff dimension, it follows
from~\eqref{5e} that
$\dim(J\backslash C)=1$
as claimed.

It remains to prove~\eqref{5e}. In order to do so, we apply
Lemma~\ref{lemma5} to a bounded subset $K$ of $J'$, say
$K=B(0,R)\cap J'$ where $R>0$. Fix $\rho$ and $\delta$ and let
$x= g_{\underline{s}}(t)\in J'$.
We want to show that there exist $r(x)$, $d(x)$ and $N(x)$ with the
properties stated in the lemma.

For $k\in\NN$ we put
\[
B_k:=B\left(f^k(x),\tfrac{1}{4}E^k(t)\right)
\] 
and note that~\eqref{5c} yields
\begin{eqnarray}\label{5JBk}
  J'\cap B_k & \subset & \Omega \cap B_k \nonumber \\
   & \subset &
   \left\{(y_1,y_2,y_3)\in \RR^3:\; 
   \tfrac{1}{2}E^k(t)<y_3<\tfrac{3}{2}E^k(t),
     y_1^2+y_2^2<\psi(y_3)^2\right\} \nonumber \\
   & \subset &
  \left[-\psi\left(\tfrac{3}{2}E^k(t)\right),\psi\left(\tfrac{3}{2}E^k(t)\right)\right]^2\times \left[\tfrac{1}{2}E^k(t),\tfrac{3}{2}E^k(t)\right] \nonumber
\end{eqnarray}
for large~$k$. This implies that $J'\cap B_k$ can be covered by
$N_k$ cubes of side length~$1$, where
\begin{equation}\label{5f}N_k\leq 5\psi \left(\tfrac{3}{2}E^k(t)\right)^2E^k(t)
\end{equation}
for large~$k$. These cubes may be assumed to lie in $H_{\geq
\frac{1}{2}E^k(t)}$. By~\eqref{2h} the preimage of such a cube
under $f$ has diameter at most~$d_k$, where
\begin{equation}\label{5g}
d_k:=\frac{2\sqrt{3}c_4\pi}{E^k(t)}.
\end{equation}
Using~\eqref{2b} we see that the diameters of the preimage of
these cubes under $f^k$ also do not exceed~$d_k$.

Let $B_0$ be the component of $f^{-k}(B_k)$ that contains $x$;
that is,
\[
B_0=(L_0\circ L_1 \circ\ldots\circ L_{k-1})(B_k)
\]
with $L_j=\Lambda^{s_j}$ as in section~\ref{constructionhairs}.
The above reasoning
shows that $J'\cap B_0$ can be covered by~$N_k$ sets of diameter at
most $d_k$, where $N_k$ and $d_k$ satisfy~\eqref{5f} and
\eqref{5g} if $k$ is large.

In order to apply Lemma~\ref{lemma5} we need to cover the intersection of~$J'$
with a ball around~$x$ and
thus we denote by~$r_k$ the radius of the largest ball around~$x$
that is contained in~$B_0$. We have to estimate $r_k$ from below.
In order to do so, let $y\in\partial B_0$ with $|y-x|=r_k$. Let~$\gamma_0$ 
be the straight line segment connecting $x$ and~$y$,
put $\gamma_j:=f^j(\gamma_0)$ for $j=1,2,\ldots,k$ 
and $B_j:=f^j(B_0)$
for $j=1,2,\ldots,k-1$. 
Then
$\gamma_k$ connects $f^k(x)$ to a point on $\partial B_k$ and
hence
\begin{equation}\label{5h}\mbox{length}(\gamma_k)\geq\tfrac{1}{4}E^k(t).
\end{equation}
It follows from~\eqref{5c} and the definition of $\Omega$ that
\begin{equation}\label{5h2}
B_k\subset B\left(0, 2 E^k(t)\right)
\end{equation}
for large~$k$.
Since $\gamma_{k-1}=L_{k-1}(\gamma_k)$, we deduce from~\eqref{2e}, \eqref{5h}
and~\eqref{5h2} that
\begin{equation}\label{5i}\mbox{length}(\gamma_{k-1})
\geq\frac{c_3}{2E^k(t)}\;\mbox{length}(\gamma_k)\geq \frac{c_3}{8}.
\end{equation}
Also,~\eqref{2h} implies that
\[
\diam B_{k-1} = \diam L_{k-1}(B_k) \leq
\frac{2c_4\pi}{E^k(t)}\diam B_k =  c_4\pi,
\]
which together with~\eqref{5c} yields
\[
\gamma_{k-1}\subset \overline{B_{k-1}}\subset 
B\left(0,2 E^{k-1}(t)\right)
\]
if $k$ is sufficiently large. Repeating the argument used to
obtain~\eqref{5i} we find that
\[
\mbox{length}(\gamma_{k-2})\geq
\frac{c_3}{2E^{k-1}(t)}\mbox{length}(\gamma_{k-1}).
\] 
Induction shows that
\[
\mbox{length}(\gamma_{k-j})\geq
\left(\frac{c_3}{2}\right)^{j-1}\frac{1}{\prod^{j-1}_{l=1}E^{k-l}(t)}\frac{c_3}{8},
\]
as long as 
$k-j+1$ 
is large enough to guarantee that
\[
\overline{B_{k-j+1}}
\subset
B\left(0,2 E^{k-j+1}(t)\right).
\] 
We put $\tau:=\frac12 c_3$ and deduce that there exist
a positive constant $\kappa$ such that
\begin{equation}\label{5j}r_k=\mbox{length}(\gamma_0)\geq \kappa\frac{\tau^k}{\prod^{k-1}_{j=1}E^j(t)}.
\end{equation}
For $N_k$, $d_k$ and $r_k$ satisfying~\eqref{5f},
\eqref{5g} and~\eqref{5j} and $\rho>1$ we have
\begin{equation}\label{5k}
\frac{N_k d_k^\rho}{r_k^3}
\leq
\frac{5\psi\left(\frac{3}{2}E^k(t)\right)^2
\left(2\sqrt{3}c_4\pi\right)^\rho}
{E^k(t)^{\rho-1} \kappa^3 \tau^{3k}}
\left(\prod^{k-1}_{j=1} E^j(t)\right)^3.
\end{equation}
Using~\eqref{5a} and the fact that
\[
\lim_{k\to\infty}\frac{\prod^{k-1}_{j=1}
E^j(t)}{E^k(t)^\varepsilon}=0
\]
for each $\varepsilon>0$ it is not
difficult to deduce that the right hand side of~\eqref{5k} tends
to 0 as $k\to\infty$. In particular, we have $N_kd_k^\rho\leq
r^3_k$ for large values of~$k$. Moreover, $d_k\in (0,\delta)$ and
$r_k\in (0,1)$ if $k$ is large. We now take such a value of $k$
and put $N(x):=N_k$, $d(x):=d_k$ and $r(x):=r_k$. Then the hypotheses
of Lemma~\ref{lemma5} are satisfied.

We conclude that $\dim(J'\cap B(0,R))\leq \rho$. With
$\rho\to 1$ and $R\to \infty$ we obtain~\eqref{5e}.
This completes the proof of Theorem~\ref{thm1}.

\section{Proof of Theorem~\ref{thm2}}\label{proofthm2}
%Let $x=x_0=(x_{0,1},x_{0,2},x_{0,3})\in J$ 
Let $x\in J$ 
and put $x_k=(x_{k,1},x_{k,2},x_{k,3}):=f^k(x_0)$
for $k\geq 0$.
We shall recursively define a sequence
$\left(y_k\right)_{k\geq 0}$ in $\RR^3\setminus J$
which has the following properties for certain positive
constants~$\eta$ and $\mu$:
\begin{itemize}
\item[(i)] $y_{k,3}=x_{k,3}$,
\item[(ii)] $\left| y_k-x_k\right|\leq 4$,
\item[(iii)] $f(y_k)\in H_{=M}$,
\item[(iv)]
$f(y_{k-1})$ and $y_k$  can be connected by a curve
$\gamma_k\subset\RR^3\setminus\left(J\cup B(0,\eta |x_k|)\right)$
with $\length(\gamma_k)\leq \mu  |x_k|$,
provided $k\geq 1$.
\end{itemize}
It is clear that we can choose $y_0$ satisfying (i), (ii) and (iii).
Suppose now that $k\geq 1$ and that $y_{k-1}$ has
been defined. We put $z_k=(z_{k,1},z_{k,2},z_{k,3}):=f(y_{k-1})$ and
note that $z_{k,3}=M$ by~(iii) and
\[
|x_k|-2a\leq \exp(x_{k-1,3})-a= \exp(y_{k-1,3})-a\leq |z_k|
\leq \exp(x_{k-1,3})+a\leq |x_k|+2a
\]
by~\eqref{2c1}. 
Now there exists $r_k=(r_{k,1},r_{k,2})\in S$
such that
$ \left|(2r_{k,1},2r_{k,2})-(z_{k,1},z_{k,2})\right|\leq 4$.
We put $u_k:=(2r_{k,1}+1,2r_{k,2}+1,M)$ and
$v_k:=(2r_{k,1}+1,2r_{k,2}+1,x_{k,3})$. Then we have
$|v_k|\geq |u_k|\geq \max\{M,|z_k|-6\}\geq \max\{M,|x_k|-6-2a\}$ and
thus for small $\eta>0$ the straight line segments
$[z_k,u_k]$ and $[u_k,v_k]$ do not intersect $ B(0,\eta |x_k|)$.
Moreover, these line segments are contained in $\RR^3\setminus J$.
We also have
\begin{equation}\label{6a}
\length([z_k,u_k])=\left|z_k-u_k\right|\leq 6
\end{equation}
while
\begin{equation}\label{6b}
\length([u_k,v_k])=x_{k,3}-M\leq|x_k|.
\end{equation}
As before we put $\underline{s}=(s_k)_{k\geq 0}:=\underline{s}(x)$ so that
$x_k\in T(s_k)$ for all $k\geq 0$. We define
$w_k:=(2s_{k,1}+1,2s_{k,2}+1,x_{k,3})$.
Then $w_k\notin J$, 
\begin{equation}\label{6c}
|w_k-x_k|\leq 4
\end{equation}
and thus $|w_k|\geq \max\{M,|x_k|-4\}$.
It is not difficult to see 
that $v_k$ and $w_k$ can be connected by a curve
\[
\sigma_k\subset H_{=x_{k,3}}\setminus\left(
J\cup B(0,\min\{|v_k|,|w_k|\})\right)
\]
which satisfies
\[
\length(\sigma_k) \leq  4\left(|v_k|+|w_k|+4\right).
\]
The lower bounds for $|v_k|$ and $|w_k|$ obtained above show that
$\sigma_k$
does not intersect the ball $B(0,\eta |x_k|)$ if $\eta$ is
chosen small enough.
We also have $|w_k|\leq |x_k|+4$ and 
\[
|v_k|\leq |u_k|+x_{k,3}\leq |z_k|+6+|x_k|\leq 2|x_k| +6+2a.
\]
Thus
\begin{equation}\label{6d}
\length(\sigma_k)\leq
4\left(3|x_k|+14+2a\right)=12 |x_k|+56+8a.
\end{equation}
Finally we connect $w_k$ to $x_k$ by a straight line
and denote by $y_k$ the point of the intersection of this
line with $\partial T(s_k)$ which is closest to~$w_k$.
Then $y_k$ satisfies (i), (ii) and (iii).
Clearly, the segment $[w_k,y_k]$ does not intersect $J$ and for small 
$\eta$ it also does not intersect $ B(0,\eta |x_k|)$.
The curve
\[
\gamma_k:=[z_k,u_k]\cup [u_k,v_k] \cup \sigma_k\cup [w_k,y_k]
\]
then connects $z_k=f(y_{k-1})$ with $y_k$ and \eqref{6a},
 \eqref{6b},  \eqref{6c}  and  \eqref{6d} yield
\[
\length(\gamma_k)\leq
13 |x_k|+66+8a.
\]
Since $x_k\geq M$ we deduce that
$\length(\gamma_k)\leq\mu |x_k|$  for some $\mu>0$.
Moreover, it follows from the definition of $\gamma_k$ that
$\gamma_k\in\RR^3\setminus\left(J\cup B(0,\eta |x_k|)\right)$
for some $\eta>0$. Thus (iv) holds.

Let now $(y_k)$ and $(\gamma_k)$ be the sequences constructed
as above. We put
\[
\Gamma_k:=(L_0\circ L_1\circ\dots\circ L_{k-1})(\gamma_k)
\]
where $L_0,L_1,\dots,L_{k-1}$ are as in section~\ref{constructionhairs}.
From~\eqref{2f} we can deduce that
\begin{equation}\label{6e}
\length(L_{k-1}(\gamma_k))\leq c_4\frac{\length(\gamma_k)}{\eta |x_k|}
\leq  \frac{c_4\mu}{\eta}.
\end{equation}
Combining~\eqref{6c} and~\eqref{6e} with~\eqref{2b} yields
\[
\dist(x,\Gamma_k)\leq 4 \alpha^k
\quad \text{and} \quad
\length(\Gamma_k)\leq  \frac{c_4\mu}{\eta}\alpha^{k-1}.
\]
It follows that 
\[
\Gamma:=\bigcup_{k=1}^\infty \Gamma_k \cup \{x\}
\]
is a rectifiable curve with endpoints $y_0$ and $x$ which except
for the point $x$ is contained in~$\RR^3\setminus J$. Thus
 $x$ is accessible from  $\RR^3\setminus J$.

\section{Examples of Zorich maps}\label{examples}
We consider Zorich maps $F(x_1,x_2,x_3)=e^{x_3}h(x_1,x_2)$ for
which there exists an annulus
\[
A:=\left\{(x_1,x_2)\in\RR^2:\; (s-\delta)^2< x_1^2+x_2^2<(s+\delta)^2\right\},
\quad 0<\delta<s<\tfrac14,
\] 
such
that if $(x_1,x_2)=(r\cos \varphi, r \sin \varphi)\in A$, then
\[
h(r\cos \varphi, r\sin \varphi)=\left(R(r)\cos \Phi
(\varphi),R(r)\sin \Phi(\varphi),\sqrt{1-R(r)^2}\right),
\]
with certain increasing and continuously differentiable functions
$R:(s-\delta, s+\delta)\to (0,1)$ and $\Phi:\RR\to\RR$ such that
$\Phi(\varphi+2\pi)= \Phi(\varphi)+2\pi$ for $\varphi\in \RR$. 
We put $t:=R(s)$, $w:=\log(s/t)$ and $a:=s\sqrt{1-t^2}/t-w$.

Then
\[
f(x_1,x_2,x_3)=F(x_1,x_2,x_3)-(0,0,a)=e^{x_3}h(x_1,x_2)-(0,0,a)
\]
maps the circle
\[
C(s,w):=\left\{(x_1,x_2,x_3)\in \RR^3:\; x_1^2+x_2^2=s^2,\;x_3=w\right\}
\]
into itself. Now
\[
\frac{\partial f}{\partial r}(r\cos \varphi, r\sin
\varphi,x_3)=e^{x_3}\left(R'(r)\cos \Phi (\varphi),R'(r)\sin \Phi
(\varphi),\frac{-R'(r)}{\sqrt{1-R(r)^2}}\right)
\] 
and hence
\begin{equation}\label{7a}
\left|\frac{\partial f}{\partial r}(s\cos \varphi,s \sin \varphi,
w)\right|=e^wR'(s)\sqrt{1+\frac{1}{1-t^2}}=\frac{s}{t}R'(s)\sqrt{\frac{2-t^2}{1-t^2}}.
\end{equation}
Also,
\[
\frac{\partial f}{\partial x_3}(r\cos \varphi, r\sin \varphi,
x_3)=e^{x_3}h(r\cos \varphi, r\sin \varphi)
\]
so that
\begin{equation}\label{7b}
\left|\frac{\partial f}{\partial x_3}(s\cos \varphi,s \sin
\varphi, w)\right|=e^w=\frac{s}{t}.
\end{equation}
We can choose~$s$ and $R$ such that $t=R(s)>4s$ and
\[
R'(s)<\frac{t}{4s}\sqrt{\frac{1-t^2}{2-t^2}}.
\]
Then the right hand sides of both \eqref{7a} and \eqref{7b} are
strictly less than $\frac{1}{4}$. 
This implies that there exists
$\varepsilon>0$ such that if $D:=\{x\in \RR^3:\dist(x,C(s,w))<\varepsilon\}$, 
then
\[
\left|\frac{\partial f}{\partial r}(x)\right|\leq
\frac{1}{4}\quad \text{and}\quad \left|\frac{\partial f}{\partial
x_3}(x)\right|\leq \frac{1}{4}
\] 
for $x\in D$.
It follows that if $x\in D$ and $y\in C(s,w)$ such that 
$|x-y|=\dist(x,C(s,w))$, then $|f(x)-f(y)|\leq \frac{1}{2}|x-y|$.
Hence
\[
\dist(f^k(x),C(s,w))\to 0
\]
as $k\to \infty$
for all $x\in D$.

The simplest choice for the function 
$\Phi$ is $\Phi(\varphi)=\varphi$. Then
$C(s,w)$ is a continuum which consists of fixed points of $f$
and attracts all points from a neighborhood of~$C(s,w)$.

We may also choose a function $\Phi$ which satisfies
\[
\Phi(\varphi)=\varphi+\varphi^3\sin\left(\frac{\pi}{\varphi}\right)
\quad\text{for}\ |\varphi|\leq \frac{1}{5}
\]
since then
\[
\Phi'(\varphi)=1+3\varphi^2
\sin\left(\frac{\pi}{\varphi}\right)-
\varphi\pi\cos\left(\frac{\pi}{\varphi}\right)\geq
1-\frac{3}{25}-\frac{\pi}{5}>0
\quad\text{for}\ |\varphi|\leq \frac{1}{5}.  
\] 
%$\Phi'(\varphi)=1+3\varphi^2\sin\left(\pi/\varphi\right)-
%\varphi\pi\cos\left(\pi/\varphi\right)\geq 1-3/25-\pi/5>0$
%$\Phi(\varphi)=\varphi+\varphi^3\sin\left(\pi/\varphi\right)$
%for $|\varphi|\leq \frac{1}{5}$. 
With $\varphi_n:=1/n$ the points
$u_n:=(s\cos \varphi_n, s\sin \varphi_n, w)$
are fixed points of $f$ for $n\geq 5$,
and since
\[
\Phi'(\varphi_n)=1-(-1)^n\frac{\pi}{n}
\quad\text{for}\ n\geq 5
\]
%$\Phi'(\varphi_n)=1-(-1)^n\pi/n$.
we see that $u_n$ is an attracting fixed point of $f$ if $n\geq 5$
is even; that is, there exists a neighborhood $U_n$ of $u_n$ such
that $f^k(x)\to u_n$ as $k\to \infty$ for all $x\in U_n$. Thus
$f$ has infinitely many attracting fixed points. If $n\geq 5$ is
odd, then $\varphi_n$ is a repelling fixed point of $\Phi$ and
thus $u_n$ is a saddle point of~$f$.

Quite generally, we can take a circle diffeomorphism
$\Psi$ and choose $\Phi$ such that the restriction of $f$ to
$C(s,w)$ is conjugate to~$\Psi$.

We thus see that the dynamics of Zorich maps can be much more complicated 
than those of exponential maps.
Presumably Zorich maps can also have ``strange attractors''.

\end{document}